\documentclass[oneside,english]{amsart}
\usepackage[english]{babel}
\usepackage{amsthm,gensymb,hyperref}
\usepackage{multicol}
\usepackage{extarrows}
\usepackage{mathtools}
\usepackage{commath}
\usepackage{amssymb}
\usepackage{amscd}
\newtheorem{theorem}{Theorem}[section]
\theoremstyle{definition}
\newtheorem{definition}[theorem]{Definition}
\theoremstyle{plain}
\newtheorem{lem}[theorem]{Lemma}
\theoremstyle{plain}
\newtheorem{prop}[theorem]{Proposition}
\theoremstyle{plain}
\newtheorem{cor}[theorem]{Corollary}
\theoremstyle{remark}
\newtheorem{rem}[theorem]{Remark}
\theoremstyle{remark}

\theoremstyle{remark}
\begin{document}

\title[Strong Hamel functions and symmetries]{Strong Hamel functions and symmetries}

\author[Bucataru]{Ioan Bucataru}
\address{Faculty of  Mathematics \\ Alexandru Ioan Cuza University \\ Ia\c si, 
  Romania}
\email{bucataru@uaic.ro}
\urladdr{http://orcid.org/0000-0002-8506-7567}

\author[Cre\c tu]{Georgeta Cre\c tu}
\address{Department of  Mathematics \\ Gheorghe Asachi Technical University \\ Ia\c si, 
  Romania}
\email{cretuggeorgeta@gmail.com}
\urladdr{http://orcid.org/0000-0003-4197-0268}

\begin{abstract}
For the geodesic spray of a Finsler space, a strong Hamel function is a Hamel function that is the geodesic derivative of a $0$-homogeneous potential function. Similarly, strong dual symmetries and strong dynamical symmetries are geodesically invariant $1$-forms and vector fields, respectively, associated with $0$-homogeneous potential functions. We prove that strong Hamel functions can be characterized in terms of strong dual symmetries and strong dynamical symmetries. We show that projective deformations by strong Hamel functions preserve the $\chi$-curvature and analyse the relationship with other classes of functions (Funk and weak Funk functions) that preserve the curvature tensors under projective deformations.
\end{abstract}

  
  \maketitle
\section{Introduction}
Hamel functions were introduced by Hamel as solutions to Hilbert's fourth problem in \cite{Hamel03}, namely as Finsler functions projectively related to the flat spray. Therefore, Hamel functions originally emerged as solutions to the Euler-Lagrange equations associated with the flat spray. Nowadays, Hamel functions represent a more general class: they are $1$-homogeneous functions that satisfy the Euler-Lagrange equations associated with an arbitrary spray.      

Hamel functions, as well as strong Hamel functions, frequently appear in Finsler geometry. A Finsler manifold has vanishing $\chi$-curvature if and only if the $S$-function is a strong Hamel function, \cite{Shen21}. Hamel functions were also used in \cite{BC20a, BC20b} to extend Beltrami's Theorem to the Finslerian setting. In these two papers, we showed that for two projectively related Finsler functions, if one has constant flag curvature, then the other also has constant flag curvature if and only if the projective deformation is a strong Hamel function.

The connection between a strong Hamel function and a geodesically invariant $1$-form (a dual symmetry) was exploited in \cite{BCC21} to provide first integrals for a class of Finsler metric.

In Theorem \ref{thm:SHF}, we prove that for a geodesic spray, strong Hamel functions induce and are uniquely determined by strong dual symmetries and consequently by strong dynamical symmetries. 

We analyse various situations in which strong Hamel functions arise. For example, in Theorem \ref{thm:schi}, we show that the $\chi$-curvature of a Finsler metric (and in particular a Finsler metric of constant flag curvature) is preserved by projective deformations if and only if the projective factor is a strong Hamel function. We provide the relationship between Hamel functions and other classes of functions: namely Funk functions, which preserve the curvature tensor, and weak Funk functions, which preserve the Ricci scalar under projective deformations. We prove that a $1$-homogeneous function is a Funk function if and only if it is a Hamel function and a weak Funk function.

In a recent paper, Crampin \cite{Crampin24} studies various relationships between symmetries (invariant vector fields for the canonical symplectic form and the energy function) and geodesic invariants (first integrals) of Finsler spaces. In this work, we focus on obtaining geodesic invariant $1$-forms (strong dual symmetries, which are not exact) from some special Hamel function. The corresponding vector field, via the symplectic structure, is a symmetry of the geodesic spray, which we call a strong dynamical symmetry, but neither a Cartan symmetry nor an exact symmetry. 

\section{Preliminaries}
In this work, we consider $M$ a connected, real and n-dimensional oriented manifold, and assume that all geometric structures are smooth. We denote by $C^\infty(M)$ the set of smooth functions on $M$, by $\mathfrak{X}(M)$ the set of vector fields on $M$, and by $\bigwedge^k(M)$ the set of k-forms on $M$.
Local coordinates on $M$ are denoted by $(x^i)$, while induced local coordinates on the tangent bundle $TM$ are denoted by $(x^i, y^i)$, for $i\in \{1,...,n\}$.

We denote by $T_0M$, the tangent bundle with the zero section removed.
On $TM$ there are two canonical structures that we will use further: the Liouville vector field and the tangent structure, given in local coordinates by $\mathcal{C}=y^i{\partial}/{\partial y^i}$ and $J=dx^i\otimes {\partial}/{\partial y^i}$, respectively. 

In this work, inspired by Crampin \cite{Crampin24}, we will use the Liouville vector field as the primary test for positive homogeneity. A geometric object $\mathcal{G}$ on $T_0M$ is positively homogeneous of degree $p$ if $\mathcal{L}_{\mathcal{C}}\mathcal{G}=p\mathcal{G}$. For example, $\mathcal{L}_{\mathcal{C}}J=-J$, hence the tangent structure $J$ is $-1$ homogeneous. 

A system of second order ordinary differential equations on $M$,
\begin{eqnarray}
\label{geodesic-system}
\frac{d^2x^i}{dt^2}+2G^i\left(x,\frac{dx}{dt}\right)=0,
\end{eqnarray}
can be identified with a special vector field on $TM$,
\begin{eqnarray*}
G=y^i\frac{\partial}{\partial x^i}-2G^i(x,y)\frac{\partial}{\partial y^i},
\end{eqnarray*}
that satisfies $JG=\mathcal{C}$. This vector field is called a semispray. If additionally, $G\in \mathfrak{X}(T_0M)$ and satisfies $[\mathcal{C},G]=G$, we say that $G$ is a spray. For a spray $G$, the functions $G^i$ are $2$-homogeneous in the fibre coordinates.

If we reparameterise the second-order system \eqref{geodesic-system}, preserving the orientation of the parameter, we obtain a new system and hence a new spray $\widetilde{G}=G-2P\mathcal{C}$.
The function $P\in C^\infty(T_0M)$ is 1-homogeneous, meaning it satisfies $\mathcal{C}(P)=P$. The two sprays $G$ and $\widetilde{G}$ are called projectively related, while the function P is called
the projective deformation factor of the spray $G$. 

Every spray induces a canonical nonlinear connection, determined by the corresponding horizontal and vertical projectors, \cite{Grifone72}:
\begin{eqnarray*}
h=\frac{1}{2}(\operatorname{Id}-[G,J]),\ v=\frac{1}{2}(\operatorname{Id}+[G,J]), \ [G,J]=h-v. \label{hv}
\end{eqnarray*}
In this paper we will use the Fr\"olicher-Nijenhuis formalism, and the corresponding commutation formulae according to \cite[Appendix A]{GM00}. Here, $[G,J]$ denotes the Fr\"olicher-Nijenhuis bracket between the spray $G$ and the tangent structure $J$. The vector valued $1$-form $[G,J]$ is $0$-homogeneous. Indeed, using the graded-Jacobi identity for the Fr\"olicher-Nijenhuis bracket of vector valued forms,
\begin{eqnarray*}
[\mathcal{C}, [J,G]]+[G,[\mathcal{C}, J]]+ [J, [G,\mathcal{C}]]=0,
\end{eqnarray*}
and considering the homogeneity of $J$ and $G$, along with $[G,J]=-[J,G]$, we obtain $[\mathcal{C}, [J,G]]=0$. Therefore, the two projectors $h$ and $v$ are also $0$-homogeneous. Locally, the two projectors $h$ and $v$ can be expressed as follows:
\begin{eqnarray}
h & = & \frac{\delta}{\delta x^i}\otimes dx^i,\ v=\frac{\partial}{\partial y^i}\otimes \delta y^i, \quad \textrm{with}, \nonumber \\
\frac{\delta}{\delta x^i} & = & \frac{\partial}{\partial x^i}-N_i^j(x,y)\frac{\partial}{\partial y^j},\ \delta y^i=dy^i+N_j^i(x,y)dx^j,\ N_j^i(x,y)=\frac{\partial G^i}{\partial y^j}(x,y). \label{delta}
\end{eqnarray}

An important geometric structure induced by a spray $G$ is the curvature tensor $R$ associated to the nonlinear connection. It is the Nijenhuis tensor of the horizontal projector, which is the semi-basic vector valued $2$-form:
\begin{equation}\label{r}
R=\frac{1}{2}[h,h]=\frac{1}{2} R_{jk}^i\dfrac{\partial }{\partial y^i}\otimes dx^j\wedge dx^k, \ R^i_{kj}=\frac{\delta N^i_j}{\delta x^k} - \frac{\delta N^i_k}{\delta x^j}. 
\end{equation}
Since the horizontal projector $h$ is $0$-homogeneous, $[\mathcal{C}, h]=0$, then $2[\mathcal{C},R]=[\mathcal{C}, [h,h]]=0$, and hence, the curvature tensor $R$ is also $0$-homogeneous.

For the geometric setting induced by a spray, we also utilize the induced dynamical covariant derivative, which is a tensor derivation defined by its action on functions and vector fields:
\begin{eqnarray*}
\nabla(f)=G(f), \ \forall f\in C^{\infty}(TM), \quad \nabla
  X=h[G,hX]+v[G,vX], \ \forall X \in {\mathfrak X}(TM).
\end{eqnarray*}
For a spray $G$ and a function $L$ on $T_0M$, we consider the following semi-basic 1-form, called the Euler-Lagrange 1-form,
\begin{eqnarray}\label{EL}
\delta_GL & = & \mathcal{L}_Gd_JL-dL=d_J\mathcal{L}_GL-2d_hL = \nabla d_JL - d_hL \\ 
& = & \left\{G\left(\frac{\partial L}{\partial y^i}\right) - \frac{\partial L}{\partial x^i}\right\} dx^i = \left\{ \nabla\left(\frac{\partial L}{\partial y^i}\right) - \frac{\delta L}{\delta x^i}\right\} dx^i. \nonumber
\end{eqnarray}
If the function $L$ is $p$-homogeneous, then the Poincar\'e-Cartan $1$-form $d_JL$ is $(p-1)$-homogeneous, and the Euler-Lagrange $1$-form is $p$-homogeneous. 

Indeed, $\mathcal{L}_{\mathcal{C}}d_JL=d_J\mathcal{C}L+d_{[\mathcal{C}, J]}L=pd_JL-d_JL=(p-1)d_JL$ and, therefore: 
\begin{eqnarray*}
\mathcal{L}_{\mathcal{C}}\delta_GL = \mathcal{L}_{\mathcal{C}}\mathcal{L}_{G} d_JL - \mathcal{L}_{\mathcal{C}}L=\mathcal{L}_{G}\mathcal{L}_{\mathcal{C}}d_JL+\mathcal{L}_{[\mathcal{C}, G]}d_JL - pdL = p\mathcal{L}_{G} d_JL - pdL = p\delta_GL. 
\end{eqnarray*} 

\begin{definition}
By a \textit{(pseudo)-Finsler metric} we mean a continuous function $F : TM \rightarrow [0, +\infty)$, satisfying the following conditions:
\begin{enumerate}
\item $F$ is smooth and strictly positive on $T_0M$.
\item $F$ is positively homogeneous of order 1, meaning that $F(x,\lambda y)=\lambda F(x,y)$, for all $\lambda> 0$ and $(x,y)\in TM.$
\item   The metric tensor with components
\[g_{ij}(x,y)=\frac{1}{2}\frac{\partial^2 F^2}{\partial y^i\partial y^j} \text{ has rank $n$ on $T_0M$}.\]
\end{enumerate}
\end{definition}
\noindent For a Finsler metric, one requires a stronger regularity condition, namely the positive-definiteness of the metric tensor $g_{ij}$. In this work, we will only use the regularity condition $3)$.  
This regularity condition is equivalent to the fact that the Poincar\' e-Cartan 2-form of the regular energy Lagrangian $L=\frac{1}{2}F^2$, $\omega_{L} = -dd_JL=g_{ij}\delta y^i\wedge dx^j$, is non-degenerate and hence it is a
symplectic structure. Therefore, the equation
\begin{equation}\label{iGddJ}
i_Gdd_JL=-dL,
\end{equation}
uniquely determines a vector field $G$ on $T_0M$, which is called the \textit{geodesic spray} of the Finsler metric.

In Finsler geometry, there are several important non-Riemannian quantities, such as the $S-$function and the $\chi-$curvature \cite{Shen21}. For a fixed volume form, $\omega=\sigma(x)dx$ on $M$, we define the $S$-\emph{function} $S$ in terms of the distortion $\tau$:
\begin{eqnarray}
S=G\left(\tau\right), \quad \tau(x,y)=\ln\frac{\sqrt{\det g}}{\sigma}. \label{sfunction}
\end{eqnarray}  

The $\chi$-\emph{curvature} has several expression, one of which is given in terms of the $S$-function:
\begin{eqnarray}
\chi=\frac{1}{2}\delta_GS=\frac{1}{2}\left\{\nabla\left(\frac{\partial
  S}{\partial y^i}\right)-\frac{\delta S}{\delta x^i}\right\}dx^i.\label{chi}
\end{eqnarray}
Alternative expressions for the $\chi$-curvature were provided by Shen in \cite{Shen21}.

\section{Strong Hamel functions}

Hamel functions were originally introduced as solutions to Hilbert's fourth problem in \cite{Hamel03}, specifically for the flat spray. Nowadays, these functions represent a larger class, consisting of solutions to the Euler-Lagrange equations associated with an arbitrary spray.   

\begin{definition} \label{dfn:Hamel}
Consider $G$ a spray. We say that a non-trivial, $1$-homogeneous function $f\in C^{\infty}(T_0M)$ is called a \emph{Hamel function} for $G$ if it satisfies the Euler-Lagrange equation $\delta_Gf=0$. 

A Hamel function $f\in C^{\infty}(T_0M)$ is called a \emph{strong Hamel function} for $G$ if there exists a $0$-homogeneous function $f'\in C^{\infty}(T_0M)$ such that $f=G(f')$. We refer to the function $f'$ as the \emph{potential function}. 
\end{definition}
The Euler-Lagrange equations for a Lagrangian function that is the geodesic derivative of another function and the relation with the Jacobi equations were studied in \cite{CM91}. 

From formulae \eqref{chi} and \eqref{sfunction}, we can see that the $S$-function is a strong Hamel function (with the distortion $\tau$ as the potential function) if and only if the $\chi$-curvature vanishes.

\begin{lem} \label{lemma:Hproj_inv}
The Euler-Lagrange $1$-form \eqref{EL} is a projective invariant, when restricted to $1$-homogeneous functions. The class of Hamel functions and strong Hamel functions are preserved by projective deformations. 
\end{lem}
\begin{proof}
Consider $f\in C^{\infty}(T_0M)$ a $1$-homogeneous function and $\widetilde{G}=G-2P\mathcal{C}$ the projective deformation of a given spray $G$. Using the fact that $f$ is $1$-homogeneous, it follows that $d_Jf$ is a $0$-homogeneous $1$-form, which means that $\mathcal{L}_{\mathcal{C}}d_Jf=0$. Therefore, we have
\begin{eqnarray*}
\delta_{\widetilde{G}}f=\mathcal{L}_{\widetilde{G}}d_Jf-df = \mathcal{L}_{G}d_Jf-df - 2\mathcal{L}_{P\mathcal{C}}d_Jf = \delta_Gf - 2 P \mathcal{L}_{\mathcal{C}}d_Jf - d_Jf(\mathcal{C})dP=\delta_Gf,
\end{eqnarray*}
and hence $f$ is a Hamel function for the spray $G$ if and only if it is a Hamel function for the spray $\widetilde{G}=G-2P\mathcal{C}$.

If there exists a $0$-homogeneous function $f'$ on $T_0M$ such that $f=G(f')$, then $f=G(f')-2P\mathcal{C}(f')=\widetilde{G}(f')$. Therefore, $f$ is a strong Hamel function for the spray $G$ if and only if it is a strong Hamel function for any projectively related spray $\widetilde{G}=G-2P\mathcal{C}$. 
\end{proof}

We now provide a characterisation of Hamel functions.
\begin{lem}\label{djdhf}
We consider $f\in C^\infty(T_0M)$  a $1$-homogeneous function. Then $f$ is a Hamel function for a spray $G$ if and only if $d_hd_Jf=0$.
\end{lem}
\begin{proof}
From Definition \ref{dfn:Hamel} and one of the equivalent expressions of the Euler-Lagrange $1$-form \eqref{EL}, we have that $f$ is a Hamel function for the geodesic spray $G$ if and  only if $d_JG(f)-2d_hf=0$. If we apply the derivative with respect to the tangent structure in the previous relation we will get that $d_Jd_hf=d_hd_Jf=0$, which is the direct implication of the result.

For the converse, will assume that $d_hd_Jf=0$. Using $hG=G$ and $J[h,G]-v$, we have that
\begin{equation*}
\begin{aligned}
i_Gd_hd_Jf&=-d_hi_Gd_Jf+\mathcal{L}_{hG}d_Jf+i_{[h,G]}d_Jf\\&=-d_hf+\mathcal{L}_{G}d_Jf-d_vf = \mathcal{L}_Gd_Jf-df=\delta_Gf,
\end{aligned}
\end{equation*} 
and hence the conclusion is obtained.
\end{proof}

In the next lemma, for the particular case when the spray is flat, we prove that Hamel functions are locally a strong Hamel function.
\begin{lem}
Consider the flat spray $G$. 
\begin{itemize}
\item[i)]
A $1$-homogeneous function $f\in C^\infty(T_0M)$ is a Hamel function for the spray $G$ if and only if the semi-basic $1$-form $d_Jf$ is locally a $d_h$-exact $1$-form. 
\item[ii)] Locally, any Hamel function is a strong Hamel function. 
 \end{itemize}
\end{lem}
\begin{proof}
i) From Lemma \ref{djdhf}, we have that the $1$-homogeneous function $f$ is a Hamel function for the spray $G$ if and only if $d_hd_Jf=0$, which means that the $1$-form $d_Jf$ is $d_h$-closed. Since $2R=[h,h]=0$, then using a Poincar\'e-type lemma (\cite[Theorem 3.2]{Stone73}) for the derivation $d_h$ we have that $d_Jf$ is $d_h$-closed if and only if it is locally $d_h$-exact, meaning that there exists a locally defined function $f'$ such that $d_Jf=d_hf'$.

ii) For the flat spray $G$ and the Hamel function $f$, there exists a locally defined function $f'$ such that $d_Jf=d_hf'$. Since $i_Gd_Jf=i_Gd_hf'$ implies $f=G(f')$, it follows that the function $f$ is locally a strong Hamel function.
\end{proof}

\section{Strong Hamel functions and symmetries}

We will now show that strong Hamel functions are uniquely determined by special geodesic invariant $1$-forms, which we will call strong dual symmetries. Via the Poincar\'e-Cartan $2$-form, these correspond to a special class of dynamical symmetries of the geodesic spray, which we will call strong dynamical symmetries. 

For a $0$-homogeneous function $f'\in C^{\infty}(T_0M)$, we consider the following $1$-form:
\begin{equation}\label{alpha}
\alpha = i_{[J,G]}\mathcal{L}_Gd_Jf' = \nabla\left(\dfrac{\partial f'}{\partial y^i}\right)dx^i-\dfrac{\partial f'}{\partial y^i}\delta y^i = \left(G\left(\dfrac{\partial f'}{\partial y^i}\right) - 2N^j_i\dfrac{\partial f'}{\partial y^j}\right)dx^i-\dfrac{\partial f'}{\partial y^i} d y^i.
\end{equation} 
The $1$-form $\alpha$ defined in \eqref{alpha} is $0$-homogeneous, which means $\mathcal{L}_{\mathcal{C}}\alpha=0$. The $1$-form \eqref{alpha}, associated to the distortion function $\tau$, was used in \cite{BCC22} to characterize first integrals for Finsler functions with vanishing $\chi$-curvature. 

For a Finsler function $F$, its energy $L=F^2/2$ is a regular Lagrangian. Therefore, for the $1$-form $\alpha$, there exists a unique vector field $X\in\mathfrak{X}(T_0M)$ such that 
\begin{eqnarray}
i_Xdd_JL=\alpha. \label{ixalpha}
\end{eqnarray}
We refer to the vector field $X$ as the symplectic dual of the $1$-form $\alpha$. Using the canonical expression of the symplectic $2$-form, $\omega_L=g_{ij}\delta y^i \wedge dx^j$, and the last two expressions \eqref{alpha} of the $1$-form $\alpha$, we obtain the following expressions of the vector field $X$:
\begin{equation} \label{alphaX}
X=g^{ij}\dfrac{\partial f'}{\partial y^j}\dfrac{\delta}{\delta x^i} + g^{ij}\nabla\left(\dfrac{\partial f'}{\partial y^j}\right)\dfrac{\partial}{\partial y^i} = g^{ij}\dfrac{\partial f'}{\partial y^j}\dfrac{\partial}{\partial x^i} + g^{ij}G\left(\dfrac{\partial f'}{\partial y^j}\right)\dfrac{\partial}{\partial y^i} .
\end{equation}
The second equality in \eqref{alphaX} is due to the fact that the horizontal distribution is a Lagrangian distribution for the symplectic $2$-form $\omega_L$, which is equivalent to the fact that the coefficients $N^i_j$ given by \eqref{delta} satisfy $g^{ij}N_j^k=g^{kj}N_j^i$, \cite{Bucataru07}.

\begin{definition}
A $0$-homogeneous $1$-form $\alpha\in \bigwedge^1(T_0M)$ is called: 
\begin{itemize}
\item[(1)]
a dual symmetry if it is geodesically invariant, namely it satisfies $\mathcal{L}_G\alpha=0$,
\item[(2)] a strong dual symmetry if it is a dual symmetry and the semi-basic $1$-form $i_J\alpha$ is $d_J$-exact, meaning that there exists a non-constant, $0$-homogeneous function $f'$ such that $i_J\alpha=-d_Jf'$.
\end{itemize}
A $(-1)$-homogeneous vector field $X\in \mathfrak{X}(T_0M)$ is called:
\begin{itemize}
\item[(3)] a dynamical symmetry if it is a symmetry of the geodesic spray, meaning $[G,X]=0$,
\item[(4)] a strong dynamical symmetry if it is a dynamical symmetry and the vector field $JX$ is a vertical gradient, meaning that there exists a non-constant, $0$-homogeneous function $f'$ such that $JX=g^{ij}\frac{\partial f'}{\partial y^j} \frac{\partial}{\partial y^i}$,
\item[(5)] an invariant vector field if $X(L)=0$.
\end{itemize}
\end{definition}
A comprehensive discussion on dual symmetries, dynamical symmetries, and their interplay in non-autonomous dynamical systems can be found in \cite{MP95}.

In the next theorem, we show that strong Hamel functions are uniquely determined by strong dual symmetries, which are uniquely determined by strong dynamical symmetries. We also clarify the relationship between the corresponding potential functions.

\begin{theorem}\label{thm:SHF}
We consider $G$ the geodesic spray of a Finsler metric $F$. Then, the following statements are equivalent:
\begin{itemize}
\item[i)] There exists a strong Hamel function for the geodesic spray $G$.
\item[ii)] There exists a strong dual symmetry for $G$.
\item[iii)] There exists a strong dynamical symmetry for $G$.
\end{itemize}
\end{theorem}
\begin{proof}
We will prove the equivalence between $i)$ and $ii)$. If $f$ is a strong Hamel function of the geodesic spray $G$, then there exists a $0$-homogeneous potential function $f'$ such that $f=G(f')=\nabla f'$ and $\delta_Gf=0$.

Using the commutation rule $d_J\nabla=\nabla d_J+d_h+2i_R$, we obtain the following equivalent expression of the $1$-form $\alpha$ introduced in \eqref{alpha}:
\begin{eqnarray}
\alpha=\nabla d_Jf'-d_vf'=d_J\nabla f'-d_hf'-d_vf'=d_Jf-df'. \label{alphaff'}
\end{eqnarray} 
We can now verify the geodesic invariance of the $1$-form $\alpha$:
\begin{eqnarray*}
\mathcal{L}_{G}\alpha=\mathcal{L}_{G}d_Jf-\mathcal{L}_{G}df'=\mathcal{L}_{G}d_Jf-dG(f')=\mathcal{L}_{G}d_Jf-df=\delta_Gf=0.
\end{eqnarray*}
From the last expression \eqref{alphaff'} of the $1$-form $\alpha$, we obtain $i_J\alpha=-d_Jf'$, which means that $\alpha$ is a strong dual symmetry.

For the converse, we consider $\alpha$ a $0$-homogeneous $1$-form on $T_0M$, which can be written as
$ \alpha=\alpha_i dx^i+\beta_i dy^i$. We now assume that $i_J\alpha = \beta_i dx^i$ is $d_J$-exact, and hence there exists a $0$-homogeneous function $f'(x,y)$ such that $i_J\alpha=-d_Jf'$.
Therefore, the $1$-form $\alpha$ is given by
\begin{eqnarray} \label{alphaf'}
\alpha=\alpha_i dx^i - \frac{\partial f'}{\partial y^i}dy^i.
\end{eqnarray}
 The Lie derivative of the $1$-form \eqref{alphaf'} along the geodesic spray $G$ is given by:
\begin{eqnarray*}
\mathcal{L}_G\alpha = G(\alpha_i)dx^i + \alpha_i dy^i - G\left(\frac{\partial f'}{\partial y^i}\right) dy^i + \frac{\partial f'}{\partial y^i} 2\left(\frac{\partial G^i}{\partial x^j}dx^j + \frac{\partial G^i}{\partial y^j} dy^j\right).  
\end{eqnarray*}
If the $1$-form $\alpha$ is geodesically invariant, then collecting the vertical components, we obtain 
\begin{eqnarray*}
\alpha_i - G\left(\frac{\partial f'}{\partial y^i}\right) + 2 \frac{\partial f'}{\partial y^j} \frac{\partial G^j}{\partial y^i}=0. 
\end{eqnarray*}
We now use the commutator: 
\begin{eqnarray*}
\left[G, \frac{\partial}{\partial y^i}\right] = \frac{\partial}{\partial x^i} - 2N^k_i\frac{\partial}{\partial y^k},
\end{eqnarray*}
to obtain the components $\alpha_i$:
\begin{eqnarray*}
\alpha_i = G\left(\frac{\partial f'}{\partial y^i}\right) - 2 \frac{\partial f'}{\partial y^j} \frac{\partial G^j}{\partial y^i} = \frac{\partial}{\partial y^i}\left(Gf'\right) + \frac{\partial f'}{\partial x^i}.
\end{eqnarray*}
The function $f:=G(f')$ is $1$-homogeneous and the $1$-form $\alpha$ is given by:
\begin{eqnarray*}
\alpha=\left(\frac{\partial f}{\partial y^i} -  \frac{\partial f'}{\partial x^i}\right) dx^i - \frac{\partial f'}{\partial y^i}dy^i = d_Jf-df',
\end{eqnarray*} 
 which coincides with formula \eqref{alphaff'}. Therefore, we have $0=\mathcal{L}_G\alpha=\delta_Gf$ and hence $f=G(f')$ is a strong Hamel function. 
 
We note that a strong Hamel function and the corresponding strong dual symmetry have the same potential function. 

Now, we prove the equivalence between $ii)$ and $iii)$. For a $1$-form $\alpha\in \bigwedge^1(T_0M)$ we consider $X$ its symplectic dual vector field, given by \eqref{ixalpha}. Therefore, we have
\begin{eqnarray*} 
\mathcal{L}_G\alpha = \mathcal{L}_Gi_Xdd_JL=i_X\mathcal{L}_Gdd_JL +i_{[G,X]}dd_JL = i_{[G,X]}dd_JL. 
\end{eqnarray*}
Since $dd_JL$ is a symplectic form, we have that $\mathcal{L}_G\alpha=0$ if and only if $[G,X]=0$, which means that $\alpha$ is a dual symmetry for the geodesic spray $G$ if and only if $X$ is a dynamical symmetry of $G$. We now show that the dual symmetry $\alpha$ is strong if and only if its dual vector field is a strong dynamical symmetry, with the same potential function. 

For a vector field, $X=X^i{\partial}/{\partial x^i} + Y^i{\partial}/{\partial y^i}\in \mathfrak{X}(T_0M)$, homogeneous of order $-1$, $JX=X^i{\partial}/{\partial y^i}$ is a vertical gradient if there exists a $0$-homogeneous function $f'$ on $T_0M$ such that $g_{ij}X^j=\frac{\partial f'}{\partial y^i}$. 

We start form the duality relation between the $1$-form $\alpha$ and the corresponding vector field $X$, \eqref{ixalpha}, we apply the algebraic derivation $i_J$, and use the fact that $i_Jdd_JL=0$. Hence, we have:
\begin{eqnarray*}
i_J\alpha=i_Ji_Xdd_JL=i_Xi_Jdd_JL - i_{JX}dd_JL=-i_{JX}dd_JL= - i_{X^k\frac{\partial}{\partial y^k}}g_{ij}\delta y^i\wedge dx^j = - g_{ij}X^idx^j.
\end{eqnarray*} 
Therefore $i_J\alpha$ is $d_J$-exact with potential function $-f'$ if and only if $JX$ is a gradient with potential function $f'$. This completes the proof of the equivalence between ii) and iii).   
\end{proof}

As a consequence of the previous calculations and theorem, we can establish further connections between strong Hamel functions (dual symmetries, dynamical symmetries), invariant vector fields and first integrals.

\begin{prop} \label{Noether}
Consider $X\in \mathfrak{X}(T_0M)$ a strong dynamical symmetry. Then $X$ and $JX$ are invariant vector fields.
\end{prop}
\begin{proof}
Consider $X$ a strong dynamical symmetry. Using the defining formulae \eqref{iGddJ} and \eqref{ixalpha}, for the vector fields $G$ and $X$, it follows:
\begin{eqnarray*}
X(L)=i_XdL = - i_Xi_Gdd_JL=i_Gi_Xdd_JL=i_G\alpha = i_G(d_Jf-df')=\mathcal{C}(f)-G(f')=0.
\end{eqnarray*}
Using the expression \eqref{alphaX} of the dynamical symmetry $X$, we obtain
\begin{eqnarray*}
JX(L)=g^{ij}\frac{\partial f'}{\partial y^j} \frac{\partial L}{\partial y^i} = \frac{\partial f'}{\partial y^i} y^i = \mathcal{C}(f')=0.
\end{eqnarray*}
Therefore, both $X$ and $JX$ are invariant vector fields for the energy function $L$.  
\end{proof}

\section{$\chi$-curvature and strong Hamel functions}

As already mentioned, the $S$-function \eqref{sfunction} is a strong Hamel function if and only if the $\chi$-curvature vanishes. Using formula \eqref{alpha} for the distortion $\tau$ and Theorem \ref{thm:SHF}, we obtain that the $S$-function is a strong Hamel function if and only if the following $1$-form is a strong dynamical symmetry, with the potential function $\tau$:
\begin{eqnarray}
\alpha=\nabla I_k dx^k - I_k \delta y^k, \quad I_k=\frac{1}{2}g^{ij}\frac{\partial g_{ij}}{\partial y^k}=\frac{\partial \tau}{\partial y^k}, \ i_J\alpha=-d_J\tau. \label{alphaik}
\end{eqnarray} 
The $1$-form \eqref{alphaik} has also been considered in \cite[Section 8]{Crampin24}, where it was shown that the corresponding vector field \eqref{ixalpha} satisfies $X(L)=0,$ hence it is an invariant vector field.

The geodesic invariance of the $1$-form \eqref{alphaik} was exploited in \cite[Lemma 3.2]{BCC22} to provide a set of $n-1$ first integrals for a Finsler manifold with vanishing $\chi$-curvature (and consequently, with strong Hamel $S$-function), as discussed in \cite[Theorem 1.1]{BCC22}. 

In this section, we explore further connections between strong Hamel functions and $\chi$-curvature.

It is well known that projective deformations preserve the curvature tensor \eqref{r} if and only if the projective factor is a Funk function. Similarly, the Ricci scalar is preserved under projective deformations if and only if the projective factor is a weak Funk function, \cite[Proposition 12.1.3]{Shen01}.

We now prove that the $\chi$-curvature is preserved under projective deformation if and only if the projective factor is a strong Hamel function. Furthermore, we establish the relationship between Funk functions, weak Funk functions, and strong Hamel functions.

We now analyse the behaviour of the $S$-function and $\chi$-curvature under projective deformations.
\begin{theorem}\label{thm:schi}
Let $G$ and $\widetilde{G}=G-2P\mathcal{C}$ be two projectively related geodesic sprays associated to the Finsler metrics $F$ and $\widetilde{F}$. For each of them, we consider the $S$-functions: $S$ and $\widetilde{S}$ associated to the same fixed volume form on $M$, and the corresponding $\chi$-curvatures: $\chi$ and $\widetilde{\chi}$. These are related as follows:
\begin{eqnarray}
\widetilde{S}=S+(n+1)P, \quad \widetilde{\chi}=\chi+\dfrac{n+1}{2}\delta_GP. \label{schiproj}
\end{eqnarray}
The $\chi$-curvature of a Finsler metric is preserved under projective deformations if and only if the projective factor is a strong Hamel function. 
\end{theorem}
\begin{proof}
For a Finsler metric $F$, the dynamical covariant derivative of its
metric tensor vanishes, $\nabla g_{ij}=0$, which means: $G(g_{ij})-g_{im}N_j^m-g_{mj}N_i^m=0$. Contracting with $g^{ij}$, we obtain:
\begin{eqnarray*}
g^{ij}G(g_{ij})=g^{ij}(g_{im}N_j^m+g_{mj}N_i^j)=2N_i^i, \quad
  \textrm{and \ hence \ } 
  N_i^i=\frac{1}{2}G(\ln (\det g) ).
\end{eqnarray*}
The two Finsler metrics $F$ and $\widetilde{F}$ being projectively
related, their geodesic sprays and nonlinear connections are connected through: 
\begin{eqnarray*}
\widetilde{G}=G-2P\mathcal{C}, \quad \widetilde{G}^i=G^i+Py^i, \quad \widetilde{N}^i_j=N^i_j+\frac{\partial P}{\partial y^j}y^i+P\delta^i_j.
\end{eqnarray*}
Taking the trace of the last formula and using the $1$-homogeneity of the projective factor $P$, we can recover this as follows: 
\begin{eqnarray*}
P= \frac{1}{n+1}\left(\widetilde{N}^i_i-N_i^i\right) =
  \frac{1}{2(n+1)}\left(\widetilde{G}(\ln\det \widetilde{g})-G(\ln\det g)\right)
   \label{p1} 
\end{eqnarray*}
Using the definition \eqref{sfunction} of the $S$-function and distortion $\tau$, for the two projectively related Finsler functions $F$ and $\widetilde{F}$, we have:
\begin{eqnarray*}
\widetilde{S}-S=\widetilde{G}(\widetilde{\tau})-G(\tau)=\frac{1}{2}\left(\widetilde{G}(\ln\det \widetilde{g})-G(\ln\det g)\right) = (n+1)P,
\end{eqnarray*}
which proves the first formula in \eqref{schiproj}.

Since the $S$-function is $1$-homogeneous, according to the first part of Lemma \ref{lemma:Hproj_inv}, we have that  
\begin{eqnarray*}
2\widetilde{\chi}=\delta_{\widetilde{G}}\widetilde{S}=\delta_G\widetilde{S}=\delta_G(S+(n+1)P)=2\chi+(n+1)\delta_GP,
\end{eqnarray*}
which gives the second formula \eqref{schiproj}.

From the second formula \eqref{schiproj} we have that $\widetilde{\chi}=\chi$ if and only if $\delta_GP=0$, hence $P$ is a Hamel function. Moreover, 
\begin{eqnarray*}
2(n+1)P=G(\ln(\det \widetilde{g}/\det g)), \label{pdet} 
\end{eqnarray*}
which shows that the projective factor $P$ is a strong Hamel function, with the potential function $\ln(\det \widetilde{g}/\det g)/2(n+1)$.
\end{proof}
We can use the result of Theorem \ref{thm:schi} to recover, as a corollary, the Finslerian version of Beltrami's Theorem, which was proven with different techniques in \cite{BC20a, BC20b}.
\begin{cor}
On a manifold of dimension $n>2$, consider $F$ and $\widetilde{F}$ two projectively related Finsler metrics and assume that $F$ has constant flag curvature. Then, $\widetilde{F}$ also has constant flag curvature if and only if the projective factor $P$ is a strong Hamel function.
\end{cor}
\begin{proof}
A Finsler metric has constant curvature if and only if the geodesic spray is isotropic and the $\chi$-curvature vanishes, \cite{LS18}.  Since the isotropy condition is projectively invariant, then the second formula \eqref{schiproj} implies that if one Finsler metric has constant flag curvature then the other also has constant curvature if and only if the projective factor is a strong Hamel function.
\end{proof}
An alternative expression for the second formula \eqref{schiproj}, along with its utility in proving a particular Finslerian version of Beltrami's Theorem in the flat case, has been derived in \cite[Proposition 1.1, Theorem 1.2]{CCQ22}.

We now establish the relation between Funk functions, weak Funk functions and Hamel functions. First, we recall their definitions.
\begin{definition}
For a given spray $G$, a non-vanishing, $1$-homogeneous function $f$ on $T_0M$ is called a:
\begin{itemize}
\item[i)] Funk function, if it satisfies $d_hf=fd_Jf$;
\item[ii)] weak Funk function, if it satisfies $G(f)=f^2$.
\end{itemize} 
\end{definition}
Next, we present the relation between these functions.
\begin{prop} \label{lem:FunkHamel}
A $1$-homogeneous function $f$ on $T_0M$ is a Funk function if and only if it is both a Hamel function and a weak Funk function.
\end{prop}
\begin{proof} 
If we start with a Funk function $f$, by contracting $d_hf=fd_Jf$ with the geodesic spray $G$, we obtain that $f$ is a weak Funk function: $Gf=f^2$. 

For a $1$-homogeneous function $f$, the Euler-Lagrange $1$-form \eqref{EL} can be written as follows:
\begin{eqnarray}
\delta_Gf=d_J(Gf)-2d_hf=d_Jf^2-2d_hf=2(fd_Jf-d_hf). \label{Funk_Hamel}
\end{eqnarray}
With the help of formula \eqref{Funk_Hamel} we now see that the equivalence $d_hf=fd_Jf$ $\Longleftrightarrow$ $Gf=f^2$ and $\delta_Gf=0$ is true. This means that $f$ is a Funk function if and only if it is a  Hamel function and a weak Funk function.   
\end{proof}

\begin{rem}
We note that for each of the two implications of Proposition \ref{lem:FunkHamel}, the condition $G(f)=f^2$ is equivalent to $f=G(\ln |f|)$. Therefore, the $1$-homogeneous function $f$ admits a potential function $f'=\ln |f|$, however this is not $0$-homogeneous. 
\end{rem} 
 
Next, we show how one can obtain examples of strong Hamel functions using projectively related Finsler metrics and some information about the projective factor, following some ideas from  \cite{Cretu20}. Consider $\widetilde{F}$ a Finsler metric projectively related to a spray $G$, such that the projective factor is a weak Funk function. It follows that $\widetilde{F}$ is a strong Hamel function.

Indeed, if $\widetilde{G}$ is the geodesic spray of the Finsler metric $\widetilde{F}$, then it is projectively related to the spray $G$, $\widetilde{G}=G-2P\mathcal{C}$. Since $P$ is a weak Funk function for the spray $G$, it follows that $G(P)=P^2$. Moreover, the projective factor satisfies $G(\widetilde{F})=2P\widetilde{F}$. Therefore, 
\begin{eqnarray*}
G\left(\dfrac{\widetilde{F}}{P}\right)=\dfrac{G(\widetilde{F})P - \widetilde{F} G(P)}{P^2}= \dfrac{2P\widetilde{F}P - \widetilde{F} P^2}{P^2} = \widetilde{F}.
\end{eqnarray*}
Since $\widetilde{F}$ is a Hamel function for the spray $G$, and using the above relation, we obtain that $\widetilde{F}$ is a strong Hamel function, with the potential $\widetilde{F}/P$.

\subsection*{Acknowledgements} We express our thanks to Mike Crampin for his comments and suggestions on this work. We are grateful to the anonymous referee, whose insightful and detailed comments helped us to improve the results and led us to correct some gaps in the proofs.


\begin{thebibliography}{0} 
\bibitem{Bucataru07} {Bucataru, I.}: \emph{Metric nonlinear connections}, Differential Geom. Appl. 25 (2007), no. 3, 335--343.
\bibitem{BC20a} {Bucataru, I; Cre\c tu, G.}: \emph{A characterisation for Finsler metrics of constant curvature and a Finslerian version of Beltrami theorem}, Journal of Geometric Analysis, {\bf 30} (1), (2020), 617-631. 
\bibitem{BC20b} {Bucataru, I.; Cre\c tu, G.}: \emph{A general version of Beltrami's theorem in Finslerian setting}, Publicationes Mathematicae Debrecen, \textbf{97(3-4)}, (2020), 439-447.
\bibitem{BCC21} {Bucataru, I.; Constantinescu, O.A.; Cre\c tu, G.}:
\emph{A class of Finsler metrics admitting first integrals}, J. Geom Phys., \textbf{166} (2021), 104254.
\bibitem{BCC22} {Bucataru, I.; Constantinescu, O.A.; Cre\c tu, G.}: \emph{First integrals for Finsler metrics with vanishing $\chi$-curvature}, Annals of Global Analysis and Geometry, 62 (2022), no. 4, 815 - 827. 
\bibitem{CM91} {Cari\~nena, J. F.; Mart\'inez, E.} \emph{Generalized Jacobi equation and inverse problem in classical mechanics}, in ``Group
Theoretical Methods in Physics'' (eds. V. V. Dodonov and V. I.
Manko), Proc. 18th Int. Colloquim 1990, Moskow, vol. II, Nova
Science Publishers, (1991) New York.
\bibitem{CCQ22} {Cheng, X. Y.; Cao, K. X.; Qing, C. Y.}: \emph{The characterizations and constructions of sprays of isotropic curvature}, Acta Math. Sin. (Engl. Ser.) 38 (2022), no. 9, 1612–1620.
\bibitem{Crampin24} {Crampin, M.}, \emph{Symmetries and geodesic invariants of Finsler spaces}, preprint 2024, DOI: 10.13140/RG.2.2.31176.29449.
\bibitem{Cretu20} {Cre\c tu, G.}: \emph{New classes of projectively related Finsler metrics of constant flag curvature},  International Journal of Geometric Methods in Modern Physics, 17(6), (2020).
\bibitem{Grifone72} {Grifone, J.}: \emph{Structure presque tangente et
    connections I}, Ann. Inst. Fourier, {\bf 22}  (1972), 287--334.
\bibitem{GM00} {Grifone, J.; Muzsnay, Z.}: \emph{Variational Principles For Second-Order Differential Equations}, World Scientific, 2000.
\bibitem{Hamel03} {Hamel, G.}: \emph{Uber die Geometrieen, in denen die Geraden die Kurzesten sind}, Mathematische Annalen 57, 2 (1903), 231-264.
\bibitem{LS18} {Li, B.; Shen, Z.}: \emph{Sprays of isotropic curvature}, 
Internat. J. Math., \textbf{29}(1) (2018),  1850003, 12 pp.
\bibitem{MP95} {Morando, P.; Pasquero, S.}: \emph{The symmetry in the structure of dynamical and adjoint symmetries of second-order differential equations}, J. Phys. A: Math. Gen. \textbf{28} (1995), 1943--1955.
\bibitem{Stone73} {Stone, A.P.}: \emph{Some remarks on the Nijenhuis tensor}, Can.J.Math., \textbf{25}(5)(1973), 903--907. 
\bibitem{Shen01} {Shen, Z.}: \emph{Differential geometry of spray and
Finsler spaces}, Springer, 2001.
\bibitem{Shen21} {Shen, Z.}: \emph{On sprays with vanishing  $\chi$-curvature}, Internat. J. Math. 32 (2021), no. 10, Paper No. 2150069, 12 pp.

\end{thebibliography}
\end{document}